\documentclass[12pt]{article}
\usepackage{color}

\def\noi{\noindent}

\def\skip{\vspace{2mm}}
\newcount\u
\def\Remark{\skip\noi{{\bf{Remark \number\u.}}} \advance\u by 1}

\usepackage{amsmath,amssymb,amscd}

\usepackage{graphics}                   
\usepackage[dvips]{epsfig}
                   
\usepackage{delarray,hhline,latexsym}

\usepackage{float}

\usepackage{enumerate}

\usepackage{fancyhdr}

\def\noi{\noindent}

\newtheorem{props}{Proposition}

\newcounter{defnctr}
\newtheorem{defns}[defnctr]{Definition}

\newtheorem{corrs}[props]{Corollary}

\newtheorem{lemms}[props]{Lemma}

\newcounter{exctr}
\newtheorem{es}[exctr]{Example}
\newenvironment{ex}{\begin{es} \rm}{\hfill $\Box$\end{es}}

\newtheorem{ess}[exctr]{Examples}

\newenvironment{proof}{\noindent{\underline{Proof} ${}$ \rm \\}}
                      {\hfill $\Box$}

\begin{document}

\title{Paradoxical decompositions and finitary colouring rules}
\author {Robert Samuel Simon and Grzegorz Tomkowicz}

\maketitle
\thispagestyle{empty}
\vfill

\noi  Robert Simon 
\newline London School of Economics
\newline 
Department of Mathematics\newline 
Houghton Street\newline 
 London WC2A 2AE\newline                    
email: R.S.Simon@lse.ac.uk
\vskip.5cm

 \noi Grzegorz Tomkowicz\\
  Centrum Edukacji $G^2$\\
        ul.Moniuszki 9\\
        41-902 Bytom\\
        Poland\\
      e-mail: gtomko@vp.pl

\vfill

\date{}

\setcounter{page}{-1}
  \noi

\newpage
\vskip2cm
\thispagestyle{empty}
\begin{center}
{\bf Abstract}
\end{center}
\vskip1cm

\noi  We  colour every point $x$  of a probability space $X$
according   to  the
colours of a finite list  $x_1, x_2, \dots , x_k$ of points
 such that each of the  $x_i$,  as a function of $x$, is 
 a measure preserving transformation.  We ask two questions
 about  a colouring  rule: 
 (1) does there exist a finitely additive extension
 of the probability measure for which the $x_i$
 remain measure preserving and also a  colouring obeying the rule almost
  everywhere that is 
  measurable with respect to this extension?,
 and (2) does there exist some  colouring obeying the rule almost everywhere?
 If the answer to the first question is no and to the second question yes, 
 we say that the colouring  rule is paradoxical. A paradoxical colouring 
 rule not only allows for a paradoxical partition of the space,
 it  requires one.  If a colouring rule is paradoxical, an axiom
 of choice is  used to prove (2) but not typically used to prove (1).  
 We show that a  form of paradoxical decomposition can be created
 from the colour classes of any such colouring.
 We show that proper vertex colouring can be paradoxical.  We conclude with several
  new topics and open questions.

\vskip2cm

\noi 2010 \emph{Mathematics Subject Classification}:
Primary: 03E05, 28C10, 37A15; Secondary: 03E25, 20E05, 43A07.\\
\noi {\bf Key words}: Axiom of Choice, Cayley graphs and graphings, 
Hausdorff Paradox,  measure theoretic paradoxes, Borel colouring

 \newpage

\section{Introduction} 

F. Hausdorff showed in [7]
 that there is no finitely additive rotation-invariant measure $\mu$ defined on all subsets of the unit sphere $\mathbb{S}^2$
 and such that $\mu(\mathbb{S}^2) = 1$. He started with some  subsets and showed that they obeyed certain rules that prevented them from being
  measurable with respect to any finitely additive rotation invariant measure. 

 In the present paper
 we reverse the above  perspective. We start with rules concerning how to colour points in a probability space according to measure preserving transformations.  
   We call such rules \emph{colouring rules}  and we are interested  in colouring
  rules
 where  the \emph{only}  sets that satisfy these rules
 cannot be measurable with respect to \emph{any} finitely additive
 $G$-invariant measure $\mu$ extending the initial
  probability measure.
  Furthermore we are not really
   interested in  such   rules that   are merely inconsistent,
  but rather in  those for
  which, using AC, the rule    can
  be satisfied in \emph {some} way. We call colouring
  rules  {\em paradoxical} when they can be satisfied
   almost everywhere but never in a finitely additive
   measurable  way as described above. The name for our rules is motivated by our Theorem 1, which is closely related to a theorem of Tarski (see [14], Thm. 11.1 and Cor. 11.2 for a proof), that shows
   that the  existence of a paradoxical colouring  rule defined
   on a standard probability space using a topologically transitive
    group action implies
   a paradoxical decomposition of the whole space or some of its subsets.
   It is worthwhile to observe that paradoxical colouring rules concern non-amenable group or semigroup actions
  and non-amenable groups may act in an amenable way (see [5] or [15]).
  Therefore it is natural to link the existence of paradoxical colouring
   rules with the following famous problem of Greenleaf: \emph{What are
     sufficient and necessary conditions for the lack of a finitely additive $G$-invariant probability measure on $X$, where $G$ is a non-amenable group acting transitively on the space $X$?} We will discuss this further
   in the third section.

   \indent Let us analyse closer the Hausdorff's construction in
           [7] that the unit sphere $\mathbb{S}^2$ can be decomposed,
           modulo a countable set $D$ (the
            set of all fixed points of non-identity rotations
            from the  appropriate subgroup
            of rotations)
            into the sets $A, B$ and $C$ satisfying the relation\\
 
 $(*)$ $A \cong B \cong C \cong B \cup C$,\\ 

  \noindent where $\cong$ denotes the congruence of sets. Here the congruences are witnessed by a subgroup of   $SO_3(\mathbb{R})$ isomorphic to $\mathbb{Z}_2 * \mathbb{Z}_3$ (and for
   convenience we will call  this subgroup  $\mathbb{Z}_2 * \mathbb{Z}_3$).  With $\sigma$ generating $\mathbb{Z}_2$
   and $\tau$ generating $\mathbb{Z}_3$, we have the requirements
    $\sigma (A) = B\cup C$,
    $\sigma (B\cup C)=A$, $\tau (A) = B$, $\tau (B)=C$, and $\tau (C) =A$.
    We call these requirements 
   the {\em Hausdorff colouring rule}.
   Measurability with respect to any
   measure where  both $\sigma$ and $\tau$ are still 
    measure preserving  would  require
    that both   half and one-third  of the space
    be taken by the set $A$.
    And  from these three sets, we can also   create six
   sets that partition the whole
    space, namely
    $A\cap \sigma (B)$, $A\cap \sigma (C)$, $\tau (A\cap \sigma (B))$,  $\tau (A\cap \sigma (C))$,  $\tau^2 (A\cap \sigma (B))$ and
    $\tau^2 (A\cap \sigma (C))$, such that by  choosing the  appropriate
     rotations  these six sets recreate 
      two copies of the whole sphere,
      modulo  the countable set $D$.   
  
     Now we  perceive the above congruences,   using
     the actions $\sigma$ and $\tau$ and the colours
      $A$, $B$ and $C$, as the composite of two
      colouring rules.  
     We colour any point $x$ according to the colours
      of $\tau^{-1} (x)$ and $\sigma (x)$. Suppose that 
     $\tau^{-1} (x)$ is coloured  already  by $B$
     and $\sigma (x)$ is coloured already by  $C$; 
      the $\sigma$  part of the rule say that
     $x$ should be coloured
     $A$ and the $\tau$  part of the rule say that $x$ should be coloured
      $C$.  How could this conflict be reconciled?
       One way 
       would be to move further afield to determine the colour of $x$.
        For example, one could specify  a new colour $E$ that can exist at
     only one point in every orbit of $\mathbb{Z}_2 * \mathbb{Z}_3$, and
     then colour the rest of the orbit by  $A$, $B$, and $C$
     according to the group element needed to move to the point in question
     from the representative coloured  $E$ (as was done in
     the original AC construction of the Hausdorff paradox). 
         However for our purposes
     such a set of   rules based on the colour $E$ 
    is  not satisfactory, 
    because  the determination of  membership in $A,B,C,$ or $E$  would involve 
     an unbounded number of group elements.

     We are interested only in  colouring rules that  are {\em finitary}
      in character, meaning that 
      the  colour of a point
      is determined by    finitely many group or
      semigroup elements.  
      Of special interest are rules that always allow for
      some  assigned colour
      no matter how descendants  are coloured
       (unlike the above  Hausdorff rules); we say that such  
        colouring  rules have   {\em rank one}.
           Rank one rules  relate  measure theoretic
           paradoxes to problems of optimisation. A determination
    of one best colour (or a subset of equivalently optimal colours)
    according to finite many  parameters is
    within the grasp of an automaton.

    Furthermore we  allow for colouring  rules that  depend on the 
     position in the space $X$. For example, let  $X$ be $ \{ 0,1\} ^G$
    where the group or semigroup  $G$ acts on $X$. 
    There could be two distinct set of rules,
    depending on whether the coordinate  $x^e$ is equal to $0$ or $1$. 
      We  identify a  special class  of colouring 
      rules: those  that do not depend on  position,
      called \emph{stationary rules}, with
      the others called  \emph{non-stationary rules}. The distinction of
      stationary vs non-stationary
      plays an important role in this paper.
 As we prefer
      colouring rules that can be preformed by an automaton,
      it should be easy to determine
       which part of the rule   should apply.  
       We require that the change in rules should
      be measurable with respect to the original probability distribution (and with all our examples it is determined by membership in clopen sets).

 There is a further distinction of whether the colours are discrete or continuous. Paradoxical colouring rules with finitely  colours are easy to find;  the challenges
      lie with the special properties these rules must obey. 
      The  existence of paradoxical colouring rules when the colours belong to a continuum and continuity of the colouring rule
       is required (with respect to location and neighboring colours) is done
      in [11] and [12].

     The primary concern  with Borel colouring
     (the questions raised in [8]); is what is 
     the {\em Borel} chromatic number,
     the least number of colours needed for a proper vertex colouring
     such that each colour class  defines a 
     Borel measurable set. 
      There are two differences  between our approach and Borel colouring. Proper colouring is only one kind of colouring rule; colouring rules may or may not imply that the colouring is proper. And 
        measurable with our approach means with respect to some  finitely additive measure.

         With a  one dimensional compact continuum of colours (a probability
       simplex determined by two extremal  colours) 
       and a Cantor set,  Simon
       and Tomkowicz [10]  demonstrated 
       a  colouring rule  using a semigroup action for which
       there is no $\epsilon$-Borel colouring for sufficiently
        small $\epsilon>0$ 
       (meaning a
        colouring function  that is Borel measurable
        and obeys the rule in all but a set of measure
        $\epsilon$), though there is a colouring that
     uses almost everywhere the two extremal  colours  of the continuum.  
     \vskip.2cm

         Ramsey Theory is a weakly related topic. A  non-amenable group $G$
          could
          act invariantly  on
          two probability spaces $X$ and $Y$ such that  there is a measurable
          surjection from $X$ to $Y$ that commutes with the actions of
           the group.  
          Assume there is  a paradoxical colouring rule for the space
          $X$ using the group $G$ however   $Y$ possesses   a finitely
          additive $G$-invariant extension measure   defined on all subsets.  
           The consequence would be that this  rule
           cannot be realised in any way on the space $Y$, though it
           would not be  contradictory in a logical  way  independent of
            the space.
            The failure of the colouring rule on part of such a  space $Y$  is
            analogous to the existence of a special  structure 
            with  Ramsey Theory.  \vskip.2cm

        The rest of this paper is organised as follows. The second
        section introduces the basic definitions. The third section 
        shows that a kind of  a paradoxical decomposition
         follows from  a  paradoxical rule  and links 
        this relationship  to the Greenleaf  problem.
        The fourth section demonstrates
        some simple  examples of paradoxical colouring
        rules. 
        The fifth  section shows that  proper vertex  colouring   can be paradoxical. 
        The  concluding
        section explores open questions and directions for further study.

\section{The Basics} 
Let $X$ be a probability space with
$\cal F$ the sigma algebra on which a
probability measure $m$ is defined. Usually $X$ will have a topology and
${\cal F}$ will be 
the induced Borel sets.

We say that the
relation $R \subseteq X \times X$ is \emph{admissible} if:\\
 
\indent $(i)$ for any $x \in X$ there are
finitely many elements $x_1,...,x_k$ of $X$
such that $(x,x_i) \in R$ for $i=1,...,k$.\\
\indent $(ii)$ for almost all  $x \in X$, $R$ is irreflexive,
meaning that $(x,x)$ is not in $R$. \\
 We call the elements $x_1, \dots , x_k$ appearing in $(i)$ the
  {\em descendents} of $x$. \\


  \indent Let $R\subseteq X\times X$ be an admissible
  relation and let  $A$ be a  set of
  colours. Most of our colouring rules have a finite $A$. If $A$ is not
  finite we require a measurable structure to $A$, namely
   a sigma algebra of subsets. 
  We say that $F: X\times A^k \rightarrow
 2^ A   $
 is a \emph{colouring rule}   on $X$
 if \vskip.2cm  (1) the graph  of $F$
  is a 
  measurable subset of $A\times (X\times A^k)$ (with respect to the  $\sigma$-algebra generated  by products
  of measurable sets using the measurability structure on $A$ and the completion of
   the probability measure on $X$),  and \vskip.2cm
  (2) all descendants $x_1, \dots, x_k$
   are measure
   preserving transformations   as functions of $x$,
   with respect to  $m$. \\
   
   It is a colouring rule of {\em rank one} if
   the correspondence $F$ is non-empty, meaning
    that for every $x$ and for every  choice
    of an element $b$ in $A^k$ the  $F (x,b)$ is non-empty.\\ 
\indent We will say that a colouring rule is \emph{deterministic} if for any $x \in X$ and $b\in A^k$ the set  $F(x,b)$ is a single point.  
 The colouring rule is {\em stationary} if
   $F(x,b)$ is determined only by $b$. The colouring rule
   is {\em continuous} if the probability space $X$ and $A$ have
   a topology, ${\cal F}$ are the Borel sets of $X$, and
  the  $F (x,b)$
 is a continuous function of  $x$ and $b$. 
  If there are finitely many colours, a  colouring rule of rank $2$ is one that is not of rank one
   and  is equivalent to
   two  colouring  rules of rank one using the same
    descendents. We 
    mean that for all $x$ and $b$ the set  $F(x,b)$  is equal to
    the intersection of the $F_1 (x,b) \cap F_2 (x,b)$ for the two 
    correspondences $F_1, F_2$ of rank one rules. With at least two
    different colours $c_1, c_2$ and no additional requirements on the correspondences
     we can always define
    $F_i (x,b)= F(x,b)$ if $F(x,b)\not= \emptyset $ and
    $F_1 (x,b) = \{ c_1\}$, $F_2 (x,b) = \{ c_2\}$ if $F(x,b)= \emptyset$.
    For the present purposes of finitely many  colours 
     we don't need ranks higher than two,
    however they may be some  contexts  where  it   makes  sense to have   
     rules of  rank higher than two.  It is possible
     that  a rank two colouring rule 
     may be implied by a rank
    one colouring rule when applied to some space, but we
    exclude this consideration from the definition of rank, which has to do with
     whether the correspondence as defined  is somewhere empty. \\
    \indent A colouring $c: X \rightarrow  A$
    {\em satisfies}
    a colouring rule $F$   if almost
     everywhere (with respect to $m$) it follows
     that $c(x)\subseteq F (x, c(x_1), \dots c(x_n))$.
      We say that a colouring rule  on $X$
     is \emph{non-contradictory} if
     it is satisfied by some colouring of the space. 
  It is \emph{paradoxical} if it is non-contradictory and for every 
   finitely additive extension $\mu$ of $m$ for which the descendants are still
   measure preserving there is no colouring $c$ of the space
   satisfying the colouring rule
   almost everywhere such that 
   $c$ is a $\mu$ measurable function (meaning that for all such $\mu$ there is a 
    measurable subset $D$ of $A$
   such that  $c^{-1}(D)$ is  not  measurable with
    respect 
    to $\mu$).

   \section{Paradoxical rules and paradoxical decompositions} We begin this section with a definition of a semigroup, introduced by A. Tarski [13] (see also [14], Chap. 10), that allows one to describe the paradoxical decompositions in an arithmetic way.\\
 \indent Let $G$ be a group acting on a set $X$. Recall that $A \subseteq X$ is $G$-equidecomposable to a subset $B$ of $X$ if there exist a partition of $A$ into the sets $A_1,...,A_k$ and elements 
 $g_1,...,g_k$ of $G$  such that $g_1(A_1),...,g_k(A_k)$ is a partition of $B$.\\ 
 \indent Define the set $X^{*} = X \times \mathbb{N}$, where $\mathbb{N}$ is the set of natural numbers and define
  the group $G^{*} = \{(g, \pi): g \in G$ \ \textrm{and} $\pi \ \textrm{is a permutation of} \ \mathbb{N} \}$. 
 Further, we define the action $G^{*}$ on $X^{*}$ as follows: $g^{*}(x,n) = (g(x), \pi(n))$.\\
 \indent Let $A$ be a subset of  $ X^{*}$.
 We will call those $n \in \mathbb{N}$ such
  that $A$ has at least one element with second coordinate $n$ the \emph{levels} of $A$.
  And we will say that $A \subseteq X^{*}$ is \emph{bounded} if it has finitely many levels.\\
 \indent Now we are ready to define the \emph{semigroup of types} $\mathcal{S}$ as follows: the set $\mathcal{S}$ is the set of equivalence classes defined by the $G^{*}$-equidecomposability
 of bounded subsets of $X^{*}$. Let $A \subseteq X^{*}$ be a bounded set, we will denote the element of $\mathcal{S}$ corresponding to $A$ by $[A]$ and call it the \emph{type of $A$}.\\
 \indent We define the addition $+$ in $\mathcal{S}$ as $[A] + [B] = [A \cup B']$, where $B'$ is a bounded set such that the levels of $B'$ and the levels of $A$ are disjoint. It is easy to check that $+$ is well-defined
 and  the element $[\emptyset]$ is the identity. Moreover, we can define an order in $\mathcal{S}$ given by $[A] \leq [B]$ if and only if there exists a $[C]$ such that $[A] + [C] = [B]$. Thus $[A] \leq [B]$ if and only
 if $A$ is $G^{*}$-equidecomposable with a subset of $B$.\\
\indent Now we can express the fact that $E \subseteq X$ is $G$-paradoxical as $[E] = 2 [E]$ in the semigroup of types. See [14] for more informations about the semigroup.\\
 \indent The following theorem proved by D. K\"onig (see [14], Thm. 10.20) is essential for our considerations:\\
 
\noi  \textbf{Theorem A.} (Cancellation Law) If $\alpha, \beta \in \mathcal{S}$ and $n$ is a positive integer then $n \alpha \leq n \beta$ implies $\alpha \leq \beta$.\\

  The proof of the Cancellation Law uses a form of the uncountable Axiom of Choice and it may happen that when applied to the algebra of Borel sets in some topological space it produces a non-Borel set.
  Therefore the problem of whether the
  Cancellation Law is true for an
  algebra of Borel sets (even in $\mathbb{R}$) remains open.

  Let $\mathcal{A}$ be an algebra of subsets of $X$. We can define an algebra $\mathcal{A}^{*}$ of bounded subsets in $X^{*}$ as the algebra with the property that each level of $A \in \mathcal{A}^{*}$ is a subset of 
 $\mathcal{A}$. Now we can define the semigroup $\mathcal{S}(\mathcal{A})$ as the set of equivalence classes restricted to the sets in $\mathcal{A}^{*}$.\\
 
 \indent We will need  the following fundamental theorem of Tarski (see [14], Thm. 11.1):\\ 

 \noi \textbf{Theorem B.} Let $\mathcal{T}$ be one of $\mathcal{S}$ or
 $\mathcal{S}(\mathcal{A})$ and let $\epsilon$ be a specified element. Then the following are equivalent:\\
  $(i)$ For all $n \in \mathbb{N}$ it is not the case that $(n+1)\epsilon \leq n \epsilon$ in $\mathcal{T}$;\\
  $(ii)$ There exists a semigroup homomorphism $\mu : \mathcal{T} \rightarrow [0, \infty]$ such that $\mu (\epsilon) = 1$.\\

 The proof of Theorem B involves the following Extension Theorem proved by Tarski [13] (Satz 1.55.):\\

 \noi \textbf {Theorem C (Extension Theorem).} Let $\mathcal{U}$ be a subsemigroup of the semigroup of types $\mathcal{T}$ and let $\epsilon \in \mathcal{U}$ be an element that satisfies $(i)$
 from Theorem B. If $\mu$ is a measure on $\mathcal{U}$ with $\mu (\epsilon) = 1$ (a semigroup homomorphism $\mu : \mathcal{U} \rightarrow [0, \infty]$ such that $\mu (\epsilon) = 1$),
 then there is an extension of $\mu$ to $\mathcal{T}$.\\

 In what follows we will use the fact that in the case of $\sigma$-algebras, the inequality of Theorem B can be substituted by equality. This follows from the Banach-Schr\"oder-Bernstein theorem (see [14], Thm. 3.6).\\
 \indent Let $(X, \mu)$ be a standard space and let $G$ be a group acting
  on $X$ such  that $\mu$ is $G$-invariant.
 We will say that a set $A \subset X$ is 
 \emph{absolutely non-measurable} if it is not measurable with respect to any finitely additive $G$-invariant measure that extends $\mu$.
 We will say also that a set $E \subseteq X$ is \emph{weakly paradoxical}
 with respect to some $G$-invariant algebra $\mathcal{A}$
 if $(n+1)[E] = n[E]$ in the semigroup
 $\mathcal{S}(\mathcal{A})$ for some positive integer $n$,
  meaning that one can pack $n+1$ copies of $E$ into $n$ copies of $E$. Notice that $(n+1) [E] =n[E]$            
 implies that $m[E] = n[E]$ for all $m>n$. What may be problematic is walking this back to $2[E] =[E]$.
  When $2[E] = [E]$, then we will say that $E$ is \emph{strongly paradoxical}.\\
 \indent In many situations the proof of existence of weak paradoxical decomposition needs some additional assumptions. But we can
 introduce the following notion that leads to a theorem that justifies the term ``paradoxical colouring rule''.
 Let $\mathcal{A}$ be the $\sigma$-algebra generated by the Borel sets and the colour classes of some colouring determined by a paradoxical rule.
 We say that a Borel subset $A$ of a standard Borel
 space $(X, \mu)$ is \emph{measurably $G$-paradoxical} if there exists a Borel set $B$ with $\mu(A) \neq \mu(B)$ and $A$ is $G$-equidecomposable 
 to $B$ using pieces in $\mathcal{A}$.\\
  \indent We have also the following notion that sharpen the definition of paradoxical rule; we say that a paradoxical finitary rule is \emph{superparadoxical}
 if any colouring satisfying the rule is not measurable with respect to \emph{any} finitely additive $G$-invariant measure defined on the $\sigma$-algebra $\mathcal{A}$.
 The motivation for this notion is that it harmonizes with the Tarski's Theorem B and as we conjecture it relates paradoxical rules to the Greenleaf problem mentioned in the introduction.\\
 \indent We mentioned above that to get more paradoxical decompositions we need some additional assumptions. Here we show that under some additional assumptions
the space exhibits the strongest form of paradoxical decompositions.\\
  \indent Let $X$ be a metric space. Consider now the type $[U]$ of a bounded open set $U \subset X$, defined for the semigroup of types $\mathcal{S}(\mathcal{A})$, where $\mathcal{A}$ is a $\sigma$-algebra of subsets of $X$.
 We will say that bounded open sets in $X$ \emph{admit sequences of similar sets} if for any bounded open set $U$ there exists a positive integer $n$ such that $(n+1)[U] = n[U]$ and there exists a sequence
 ${V_i}$ of open sets with diameters decereasing monotonically to $0$ and such that $(n+1)[V_i] = n[V_i]$. An example of a metric space where open sets admit sequences of similar sets
 provides Euclidean space. To observe it we use some conjugations by similarities.\\       
 \indent Since the finitely additive measures appearing in the definition of paradoxical rules are not a priori unique, we need a kind of notion that
 assures the uniqueness. These topic is related to some properties of groups acting on the space in question.\\ 
 \indent Let $\mathbb{S}^n$ be the $n$-dimensional Euclidean sphere. Recall that for $F \in L^2(\mathbb{S}^n, \lambda)$ with $\lVert F \rVert_2 = (\int_{\mathbb{S}^n} |F|^2  d \lambda)^{1/2}$ 
 the action of a countable group of isometries $G$ on $\mathbb{S}^n$ has spectral gap if there exist a finite set $S \subset G$ and a constant $\kappa > 0$ such that
 $\lVert F \rVert_2 \leq \kappa \sum _{g \in S} \lVert g \cdot F -F  \rVert_2 $ for any $F \in L^2(\mathbb{S}^n, \lambda)$ with $\int_{\mathbb{S}^n} F d \lambda = 0$. Clearly, the notion of spectral gap can be extended to probability spaces.\\
 \indent Recently R. Boutonnet, A. Ioana and A. S. Salehi Golsefidy [3] generalized the notion of spectral gap as follows: 
 Let $G$ be a countable group acting on a standard measure space $X$ then the action has \emph{local spectral gap} with respect to a measurable set $B \subset X$ of finite measure
 if there exist a finite set $S \subset G$ and a constant $\kappa > 0$ such that
 $\lVert F \rVert_{2, B} \leq \kappa \sum _{g \in S} \lVert g \cdot F -F  \rVert_{2,B} $ for any $F \in L^2(\mathbb{S}^n, \lambda)$ with $\int_{B} F d \lambda = 0$. Here $\lVert F \rVert_{2, B}$ denotes $(\int_{B} |F|^2 d \lambda)^{1/2}$.\\  
  \indent The notion of local spectral gap appears implicitely, in the case of action of group of isometries on $\mathbb{R}^n$, in Margulis [9] and is used to get a positive answer to Ruziewicz problem for $\mathbb{R}^n$ $(n \geq 3)$.\\
  \indent Let $X$ be a metric space and let $G$ be a group acting on $X$, we will say that the action of $G$ is \emph{paradoxical} if any two bounded subsets of $X$ with compact closures and non-empty interiors are $G$-equidecomposable 
 with respect to the algebra of all sets.\\
 \indent It is well-known that if the action of a group $G$ is paradoxical and has spectral gap than the finitely additive extension of the completed Borel measure is unique (see [3],[6]). Therefore we have a broad
 class of spaces where the uniqueness of our extension holds.\\

\noi \textbf{Theorem 1:} \emph{ Let $F$ be a paradoxical rule with finitely many colours defined by a
  topologically transitive
  action of a  group $G$ on a
   standard Borel space $(X, \mu)$ with $G$-invariant measure $\mu$. Then:\\
 $(i)$ \ \ For any colouring $c$ satisfying the rule $F$
    there exists a Borel set that is measurably $G$-paradoxical with respect to the $G$-invariant $\sigma$-algebra $\mathcal{A}$ generated by the Borel sets
 and all colour classes of $c$;\\
 $(ii)$ \ \ Assuming additionaly that the extension of $\mu$ to $\mathcal{A}$ is unique, for any colouring $c$ that satisfies $F$,
   any bounded open subset of $X$ is weakly $G$-paradoxical with respect to $\mathcal{A}$; \\
 $(iii)$ \ \  Moreover, suppose that $X$ is metrizable and open subsets of $X$ admit sequences of similar sets. Then 
   any compact set of $\mathcal{A}$ with non-empty interior is strongly paradoxical using pieces in $\mathcal{A}$;\\
 $(iv)$ \ \ If the rule $F$ is superparadoxical, then any compact subset of $X$ with non-empty interior is weakly $G$-paradoxical in $\mathcal{A}$.}\\

 \emph{Proof of part (i).} It is easy to see that any measure on the semigroup of types $\mathcal{S}(\mathcal{A})$ induces a finitely additive $G$-invariant measure $m$ on the corresponding
 algebra $\mathcal{A}$. Indeed, for any subset $A \in \mathcal{A}$ we put $m(A) = \nu ([A])$, where $\nu$ is the measure defined on $\mathcal{S}(\mathcal{A})$ and $[A]$ is the type of $A$.\\
 \indent On the other hand, any finitely $G$-invariant measure $m$ defined on $\mathcal{A}$ induces a measure on $\mathcal{S}(\mathcal{A})$. To get such a measure it is enough to sum up the $m$-measures of
 finitely many levels defining any type of $\mathcal{S}(\mathcal{A})$.\\
 \indent Suppose to the contrary that no Borel set is measurably $G$-paradoxical with respect to $\mathcal{A}$. Then the Borel measure $\mu$ induces a measure on the subsemigroup $\mathcal{U}$ of $\mathcal{S}(\mathcal{A})$ that
 corresponds to the types of Borel sets. Clearly, the measure is well-defined by the assumption that no two Borel sets are measurably $G$-paradoxical with respect to $\mathcal{A}$. Thus, by the Extension Theorem,
 there exists a measure on $\mathcal{S}(\mathcal{A})$ that extends the measure on $\mathcal{U}$. Clearly, the measure induces a $G$-invariant measure on $\mathcal{A}$ that extends $\mu$.
  Therefore the colouring $c$ is measurable with respect to $m$. But this contradicts the fact that $F$ is a paradoxical rule and finishes the proof. $\Box$ \\     

 \emph{Proof of part (ii).} Suppose to the contrary that there exists a bounded open subset of $X$ that is not weakly $G$-paradoxical with respect to $\mathcal{A}$. 
 Then Theorem B and the uniqueness of the extension of $\mu$ as the Borel measure
 imply the existence of a finitely additive, $G$-invariant measure $\nu$ on $\mathcal{A}$ with $\nu (U) = 1$.
 And this in turn implies that the colouring $c$ is $\nu$-measurable, in contradiction to the paradoxicality of the rule $F$. $\Box$\\

  \emph{Proof of part (iii).} First, we observe that since bounded open sets are weakly paradoxical and that $\mathcal{A}$ is a $\sigma$-algebra,
  there exists a positive integer $m$
 such that $(m+1)[U] = m[U]$ in $\mathcal{S}(\mathcal{A})$ for any open set $U \subseteq X$. And this implies:\\

   $(1)$ \ \ \ \  $k[U] = m[U]$ in $\mathcal{S}(\mathcal{A})$ for any integer $k \geq m$. \\

   Now take, a compact subset $K \subset X$ with nonempty interior and let $V \subset K$ be an open set. Since the action of $G$ is topologically transitive,
   $K$ can be covered by the union of
   $l$ congruent copies of $V$. Therefore we have:\\

  $(2)$ \ \ \ \  $[K] \leq l[V]$ in $\mathcal{S}(\mathcal{A})$ for some integer $l \geq 2$. \\

 Clearly, $(2)$ implies:\\

  $(3)$ \ \ \ \  $m[K] \leq ml[V]$ in $\mathcal{S}(\mathcal{A})$ for some integer $l \geq 2$, \\

 and by $(1)$ and $(3)$ we get\\

 $(4)$ \ \ \ \  $m[K] \leq ml[V] = m[V]$ in $\mathcal{S}(\mathcal{A})$. \\

 On the other hand, since $V \subseteq K$, we get $m[V] \leq m[K]$ which by the Banach-Schr\"oder-Bernstein theorem and by $(4)$ implies $m[K] = m[V]$ in
  $\mathcal{S}(\mathcal{A})$.\\
 \indent Since open sets admit sequences of similar sets, there are is an open set $W$ such that that $(m+1)[W] = m[W]$ and the union of $m$, $G$-congruent disjoint copies of $W$ is contained in $K$.\\
 \indent Hence, by the Banach-Schr\"oder-Bernstein theorem and by $(4)$ we get $m[K] = m[W]$ and also $m [K] \leq [K]$ in $\mathcal{S}(\mathcal{A})$. Clearly, $ [K] \leq m[K]$ is also true.
 Again, the Banach-Schr\"oder-Bernstein theorem imples $ [K] = m[K]$ in $\mathcal{S}(\mathcal{A})$.\\
 \indent Finally, to get the strong paradox we observe that $[K] \leq 2[K] \leq m [K]$ in $\mathcal{S}(\mathcal{A})$ and so the Banach-Schr\"oder-Bernstein theorem imples $[K] = 2[K]$ in $\mathcal{S}(\mathcal{A})$. $\Box$\\ 

 \emph{Proof of part (iv).} This follows in exactly the same manner like in the cases $(ii)$ and $(iii)$ taking into an account that any colouring $c$ is not measurable with respect to \emph{any} finitely additive
 $G$-invariant measure defined on $\mathcal{A}$. $\Box$ \\

   In the context of the above
   we may ask  the following question
    such that a positive answer would extend the results of Tarski:\\

   \noi  \textbf{Question 1:} Let $G$ be a group acting in a measure preserving way on a standard Borel probability space $(X, \mu)$. Is the existence of a $G$-paradoxical decomposition of
    a subset $E$ of $X$ with $\mu(E) > 0$ equivalent to the existence of a
    paradoxical colouring  rule using the group $G$ to define
     the descendents?  \\

     M. Bounds [2] obtained recently a partial solution to this question. Whenever there is a partition of the space $X$ witnessing $n[X] = (n+1) [X]$ he showed that there is a paradoxical
     colouring rule for which that partition is a solution.
      We note that a positive answer to Question 1  may  shed some light on the following problem of Greenleaf: Let $G$ be a group acting faithfully and transitively on a space $X$. What are neccessary and sufficient conditions for the
 action such that there is no finitely additive $G$-invariant probabilistic measure on all subsets of $X$?\\

\noi  \textbf{Question 2:} Let $F$ be a paradoxical rule and let $c$ be a colouring satisfying the rule. Are some of the colour classes
  of $c$ absolutely non-measurable?\\

  A positive answer to this question would lend relevance to the above mentioned result of M. Bounds.
   The following construction that shows that the pieces witnessing the Banach-Tarski paradoxical decomposition are absolutely non-measurable with respect to any $\sigma$-algebra that contains the pieces. It 
 indicates the answer to Question 2 might be positive:\\

 Let $F$ be a free non-abelian group of rank two generated by $s$ and $t$. 
 Consider the sets $w(s), w(s^{-1})$ and  $w(t), w(t^{-1})$ of reduced words of $F$ begining on the left on $s, s^{-1}$ and $t, t^{-1}$, respectively.
 Then we have:\\

 $(*)$ \ \ \ $s w(s^{-1}) = w(s^{-1}) \cup w(t) \cup w(t^{-1})$ and also $t w(t^{-1}) = w(t^{-1}) \cup w(s) \cup w(s^{-1})$.

 \begin{props}
   \label{props:1}
\emph{ Let $\mathcal{A}$ be the $G$-invariant $\sigma$-algebra generated by the set $w(s^{-1})$. Then  $w(s), w(t), w(t^{-1}) \in \mathcal{A}$.  }
 \end {props} 

 \emph{Proof.} By $G$-invariance of $\mathcal{A}$ we obtain that $s w(s^{-1}) \in \mathcal{A}$ and thus $w(s) \in \mathcal{A}$. Since $\mathcal{A}$ is closed under complementation we get that
 $w(t^{-1}) \cup w(t) \cup \{1\} \in \mathcal{A}$. Now, by $G$-invariance we get that $t(w(t^{-1}) \cup w(t) \cup \{1\}) = w(s^{-1}) \cup w(s) \cup w(t^{-1}) \cup tw(t) \cup \{t\} \in \mathcal{A}$ too. 
 Since $w(s^{-1}) \cup w(s) \in \mathcal{A}$, then subtracting it from $w(s^{-1}) \cup w(s) \cup w(t^{-1}) \cup tw(t) \cup \{t\}$ still leaves us in $\mathcal{A}$. So $w(t^{-1}) \cup tw(t) \cup \{t\} \in \mathcal{A}$.
 Iterating, the process of moving by $t$ and subtracting then by $w(s^{-1}) \cup w(s)$ we obtain that the set $w(t^{-1}) \cup t^nw(t) \cup \{t^n\} \in \mathcal{A}$ for any positive integer $n$.\\
 Now, since $\mathcal{A}$ is a $\sigma$-algebra and the group $F$ is free we have that 
 $$\bigcap_{n=1} ^{\infty} (w(t^{-1}) \cup t^nw(t) \cup \{t^n\}) = w(t^{-1}) \in \mathcal{A}.$$

 Now, by the $G$-invariance we infer that $tw(t^{-1}) \in \mathcal{A}$ and so $w(t) \in \mathcal{A}$ too. $\Box$\\

 \section{Easy Examples} The Hausdorff rule is of rank two.
  We need to  establish  that
 there are paradoxical  rules of rank one. The first example
 is  a rank one stationary rule with three
 colours   whose satisfaction is equivalent to the   
    satisfaction of   the Hausdorff rule. 
    Example \ref{ex:1}  was  given to us by
  M. Bounds [2].
  The second  and third  examples are rank one rules that mimick rank two rules. 
   We are not so interested in
  the second and third    examples because they represent an
   oversimplification, a  cutting of 
   a Gordian knot using an enlarged space or an enlarged
   colour set.\\
    \indent It is worthwile to note that the colour classes  of any colouring satisfying the rule from Example \ref{ex:1}  form a paradoxical decomposition of $X$. In particular,
 we can show that the unit sphere $\mathbb{S}^2$ up to a countable set of fixed points is paradoxical with sets defined by any of the colourings satisfying our rule.
 This in turn implies (by the arguments in Chapter 3 in [14]) a paradoxical decomposition of Banach and Tarski of the unit ball in $\mathbb{R}^{3}$.\\

 \begin{ex} \label {ex:1}  
   Let $G$ be any group generated
   freely  by $\tau, \sigma_1, \dots , \sigma_{k-1}$, where
   the order of $\tau$ is either infinite or divisible by $3$,
   and the order of all the $\sigma _i$ are either infinite or
    divisible by $2$. Let $X$ be a  probability
    space acted upon freely by $G$ almost everywhere  such
    that every element of $G$ is measure preserving. 
 There are three colours $A_1, A_2, A_3$ and the arithmetic
 is modulo $3$. Assume that $\tau^{-1} (x)$ is coloured $A_i$.
 If $\tau(x)$ is coloured either $A_i$ or $A_{i+1}$
 then colour $x$ with $A_{i+1}$. If $\tau (x) $ is coloured $A_{i-1}$
 then count how many points among $\tau (x), \tau^{-1} (x), \sigma _1 (x),
 \dots, \sigma _{k-1}(x)$ are coloured $A_1$ (which cannot 
 be all of them because $\tau (x)$ and $\tau ^{-1}(x)$ are coloured 
 differently). If the number is strictly
 between $0$ and $k$, colour $x$ with the same colour as that of 
 $\tau (x)$. Otherwise $x$ is
 coloured with $A_{i+1}$. \end{ex} 
 
 \begin{props} \label {props:1}  The colouring rule of Example 1 is paradoxical.\end{props}

  {\bf Proof:} If the colouring $c$ satisfies the colouring rule,
 we claim that for almost all $x$ if $\tau^{-1} (x)$ is coloured $A_i$
 then $x$ is coloured $A_{i+1}$.  The colour given to $x$ by the rule
 is never the same as the colour given to $\tau^{-1}(x)$.
 The only way for $x$ to be not coloured with $A_{i+1}$
 is for  $\tau (x)$ to be  coloured $A_{i-1}$.
 But then   $x$ is also assigned the colour
 $A_{i-1}$, which is not possible
  because then $\tau (x)$ would not  have
  been assigned the colour  $A_{i-1}$. 
  Now that we know that the colours cycle forward in the $\tau$ direction,
  by the rule it must follow that  either
   none of  the $\tau (x), \tau^{-1} (x), \sigma _1 (x),
 \dots, \sigma _{k-1}(x)$ are coloured $A_1$ 
  or all but one are coloured $A_1$. If $x$ is coloured $A_1$ 
 then because both $\tau ^{-1} (x)$ and $\tau (x)$ cannot be coloured 
 $A_1$ it must follow that none  of the  $\sigma _1 (x),
 \dots, \sigma _{k-1}(x)$ are coloured $A_1$. If $x$ is not coloured $A_1$ 
 then at least one of $\tau ^{-1} (x)$ and $\tau (x)$ is coloured 
 $A_1$, hence all of  the  $\sigma _1 (x),
 \dots, \sigma _{k-1}(x)$  are coloured $A_1$.

  Assuming that the colouring is measurable with respect to some
 finitely additive $G$-invarient
 extension measure, it
  follows from the measure preserving properties of all
 the generators that the global probability for $A_1$ must be
 simultaneously at most $\frac 13$ and at least $\frac 23$,
 a contradiction. 
 Next we show that the colouring rule can be satisfied.
 
 Choosing an $x$ in  any  orbit of $G$ (using some axiom of choice)  for
  which $g_1 (x) = g_2 (x)$ implies that $g_1 = g_2$,  colour $x$ with 
  $A_1$, then colour the $\tau$ orbit of $x$  accordingly. Then
  continue with the rule that if $y$ is coloured $A_1$ then
  all of the $\sigma_i (y)$ and $\sigma^{-1} (y)$  are not coloured $A_1$ and
  if $y$ is not coloured $A_1$ then all
  of the $\sigma _i (y)$ and $\sigma_i^{-1}$
    are coloured $A_1$. Continue the process indefinitely on the whole orbit. \hfill $\Box$  
 \\

\begin{ex} \label{ex:2} 
Assume that $R$ is a paradoxical rule
using a semigroup $G$ acting on $X$.
Let $\hat X$ be 
    $X\times  \{ a,b\}$, with the probability 
    given to both  $\{a\} \times A$ and  $\{b\} \times A$ half of
     that given to $A$ in the space $X$. 
     We define $\rho$ to be the measure preserving
     involution that switches between $(a,x)$ and $(b,x)$
     for every $x\in X$. 
     We use the same colours, and assume that there are at least two colours.
      The new rule    assigns  to
    every  $(b,y)$ the same colour as that of $(a,y)$.
     The rule for colouring $(a,y)$ 
    is   easy to describe.   If the colour
     of $(b,y)$ follows the rule $R$ with respect
     to the  $(a,z_i)$ for all the descendents $z_i$ of $y$,
     as it should be after dropping the $a$ and $b$ coordinates,
      then
    colour $(a,y)$ the same colour as that of $(b,y)$.
    Otherwise colour $(a,y)$ differently than
     $(b,y)$.    
    In any  colouring satisfying this new  rule the points   
    $(a, y)$ and $(b,y)$ are given the same colour, implying
    that the $R$  rule is  being  followed on both
    copies of $X$.\end{ex} \vskip.2cm 

    \begin{ex} \label {ex:3} 
  Assume  that $R$ is a paradoxical rule
  using a semigroup $G$ acting on $X$ where one of the descendents is  defined
   by an invertible $g\in G$.
   Let $C$ be the colouring set of the rule $R$ and instead
    colour $X$ with $C^2$ colours. 
    The new rules requires that 
     the second colour of  $x$ is the 
    copy of first colour of $g (x)$, and the first colour of $x$ agrees
    with the second colour of $g^{-1} (x)$ if and only if the
    second  colour of $g^{-1} (x)$ is the correct colour for $x$ when
    following the rule $R$  with respect to the
    first colours of descendents of $x$. The process is really the same
    as that of Example \ref {ex:2} , with the second colour of $g^{-1} (x)$ playing
     the same role as the colour of  $(b,x)$ in Example \ref{ex:2} . 
    \end{ex}

     \section{Proper Colouring}

       In this section 
     we  present a context in which the proper vertex colouring rule  is paradoxical. 
 The distinction between stationary and non-stationary rules is especially important with proper colouring. With proper colouring, we make no  requirement of  the  colouring
   other than that the colour of $x$ and the colour of any
   descendent of $x$ must be different.  With other colouring
   rules, we could keep the finite set of descendents fixed  and allow some of them
   to  be irrelevant, dependent on location.
   But with proper colouring all desecendents must be relevant and in the same way. The best way to resolve this is
   to define  our proper colouring rule  to be    non-stationary if the subset of  descendants relevant to proper colouring  changes with location. For one, this
    allows the degree of the relevant graph to vary by location. If all
    descendents are relevant everywhere to the proper colouring  then it is stationary, and free action almost everywhere implies
     that almost everywhere the degree of the vertices  is a constant.
     If the rule is non-stationary, we do require the Borel property, meaning in the context
     of finitely many colours and finitely many descendents that the subset of relevant descendants is determined by membership in finitely many Borel sets
      that partition the space.

    Below we present an example where   non-stationary proper colouring  is paradoxical.  
      Before we do this, we first present a paradoxical  colouring rule where some  colour is always allowed, meaning rank one.
      We move from this to a proper  list colouring that is paradoxical, and finally
       to a non-stationary proper colouring that is paradoxical.    
       \vskip.2cm
        \noi 
\subsection { The fundamental example:}

\begin{ex} \label {ex:i} 
 With $X= \{ -1,1\} ^{\mathbb{F}_2}$,  
 let  $T_1$ and $T_2$ be the two generators of $\mathbb{F}_2$.
 We  create four colours 
 $a_{1, u} $, $a_{1, c} $, $a_{2,u} $ and $a_{2,c} $, the $c$ or $u$ standing for  ``crowded'' or ``uncrowded''. The colour is broken into two parts, its
  active part, the $1$ or $2$, and its passive part, the $c$ or $u$. 
  From every $x\in X$ we must direct  an arrow
  from  $x$ to one of two of its  neighbours. If $x^e=-1$, our choice is
  between $T^{-1} _1x$ or $T^{-1} _2x$. If $x^e=1$ our choice
  is between $T_1x$ or $T_2 x$. 
  If both of our choices for such neighboring points
  are coloured with the passive colour
  $c$ (meaning either $a_{1,c} $ or $a_{2,c}$)
  or both of our choices are coloured with  $u$,  the colouring
    rule allows us 
    to place the arrow in either direction. 
    But if one of these two points
     is coloured with $c$ and the other is coloured with 
    $u$ then the colouring rule demands that
     the arrow must be placed toward the one coloured with  $u$. 
     If two or more arrows are directed toward a point $x$,
     then the colouring rule requires that
     its passive  colour is   $c$. Otherwise its passive
      colour must be   $u$.
      The active colour $1$ or $2$ refers to the direction
       of the arrow, the active colour $i$ if the arrow is directed 
       to  either  $T_i$ or $T^{-1} _i$. Notice that every point receives
       an active and a passive colour and that the colouring scheme
       pertains to two independent systems on $X$ that are separated
        by alternating applications of the generators. 
  \end{ex}

  \noi  We define the p-degree of a vertex
   as the number of potential arrows that could be pointed
   toward this vertex  (the {\em passive} degree,
  to distinguish from the usual definition of  degree in  a 
    graph).

       \noi
       \begin{props} \label {props:t}  The colouring rule of
         Example \ref{ex:i}  is paradoxical.
\end{props}

       \begin{proof}  
    The vertices with p-degree zero   take up $\frac 1 {16}$ of the space.  
    Our claim is that with a colouring satisfying the rule
    almost all vertices are uncrowded, meaning they  have the
    passive colouring $u$. 
    This implies that the colouring cannot be measurable with respect
    to any finite extension for which the group $\mathbb{F}_2$ is measure
     preserving, 
   because  at least $\frac 1 {16}$ of the vertices  cannot have
   any arrows pointed  toward them. By one accounting
   there is an arrow exiting every vertex but by another
   accounting the average number of arrows coming in to vertices
    must be no more than
    $\frac {15} {16}$. Another way to see the paradox is that the arrows define
    an almost everywhere injective mapping via group elements from the
     whole space to a cylinder set of measure no more than $\frac {15} {16}$.

       \vskip.2cm

       \noi
   Now let us assume that $x$ is a crowded vertex and see what is necessary
   to maintain this situation in a colouring satisfying the rule.
   There must be two distinct vertices $y$ and $z$ such that
   there is an arrow from $y$ to $x$ and an arrow from $z$ to $x$. Lets
   focus on just one of them, without loss of generality the $y$. As the
   colouring rule is satisfied,
   the existence of an
    arrow from $y$ to $x$ implies that the other  vertex
    toward which  $y$ could direct an arrow
     is another crowded vertex, call it $w$.
   As the arrow
   is already defined from $y$ to $x$
   it means that there are at least  two arrows
   pointed inward to  $w$ that do not start at $y$. Letting $v_1$ and $v_2$
   be two of those vertices, let $u_1$ and
   $u_2$ be the vertices for which there could have  be an arrow
   from $v_i$ to $u_i$, however instead the arrow was
   from the $v_i$ to $w$. We recognise by induction
   the existence of a chain of backwardly directed
    arrows, starting at $w$, moving to
    $u_1$ and $u_2$ (from  vertices so far un-named)  and beyond, such that
    at alternating stages 
    the induced graph branches into either two or three directions. 
   If there are three
    such branches in some places, we could reduce
    the problem to  the existence of a binary tree. Each of these
   vertices of the binary tree  has p-degree at least three. 
   Now let $p$ be the probability of there existing
   such an infinite chain, the probability relative
   to the start at a   vertex like $y$
   moving in the direction  away
    from  $x$. 
As the space is
defined homogeneously (that the probabilities for $-1$ or $1$ are
independent regardless of  shift distances)
we can calculate $p$ recursively. From the start, at $x$ and  $y$,  
there are two possibilities: the next vertex $w$  could be of p-degree
three and the chain of backward arrows
continues with these two adjacent vertices on the other side of
$y$,  or $w$ is of p-degree four and 
the chain continues with at least two of these three vertices on
the other side of $y$.
In the first case the conditional  probability that $w$ is of p-degree three
is $\frac 38$ (conditioned on the move from 
      $y$ to $w$) and then of continuation of the chain indefinitely 
happens with probability $p^2$. In the second case,
the conditional  probability of p-degree four is
$\frac 18$ and the probability of continuation $3p^2 - 2p^3$
(three choices for  the two next vertices  minus the
possibility, counted twice,
that continuation is possible in all three directions).
We have the
formula $p = \frac 38 p^2 + \frac 18  (3p^2 - 2p^3)= \frac {3p^2 - p^3} 4$.
Factoring out the $p=0$ solution, we are left
with $p^2 -3p +4$, which has no real solutions.
We conclude that the  stochastic structure of p-degrees  does not allow
     for crowded points to exist in more than a set of measure zero.

       \vskip.2cm

       \noi     
     Now we show (using AC) that there does exist a  colouring satisfying
      the rule. For every orbit choose a representative $x$ and label every other vertex in this same orbit as $gx$ by the group element $g$ used
      to travel from $x$ to $gx$. Every group element $g$ has a length, the
      minimal  number of uses of $T_1$, $T_1^{-1}$, $T_2$ and $T_2^{-1}$
      used to construct $g$ (where the length of
      the identity is zero). Let the length of a vertex $gx$ in the orbit be the length of $g$. Colour $x$ with either colour $1$ or $2$.
       Colour all  vertices of length $1$ next, then all
      vertices of length $2$, and so on in the following way. At every stage of the process, a vertex of length $l$ is adjacent to one vertex
     of length $l-1$ and three of length $l+1$. Therefore from any vertex $y$ of length $l$ one can always point the arrow toward a vertex of length $l+1$, regardless of the value
       of $y^e$. There can be no other arrow toward $y$, as the three other points adjacent to $y$ are all  of length $l+2$.
      Seeing that there are no vertices receiving two arrows, all vertices can be coloured with either $a_{1,u} $ or $a_{2,u} $. \end{proof}

       \vskip.2cm

       \subsection {  Proper list colouring:}

\begin{ex} \label {ex:l} 
       \noi  Consider the four group elements
       $T_1 T_2^{-1}$, $T_1^{-1} T_2$, $T_2 T_1^{-1} $, $T_2^{-1} T_1$, and the
 twelve  length four group elements created by multiplying two of them together.
 Let $g_1, \dots , g_{16}$ be
 these distinct $16$  group elements,
 with $g_1, \dots , g_4$ the ones of length $2$ and
 the $g_5, \dots, g_{16}$ the ones of length $4$.
  Let $X'$ be the subset of $X$ such that $g$ acts freely on $X'$. Notice that these 16 group elements are paired up by inverses. 
  By (see [1] or [4]) there are
  $17$ Borel sets, partitioning the space $X'$,
  such that $g_i x$ and $x$ belong to different sets 
  for every choice of $i$ and $x\in X'$. To each of these $17$ sets
   associate a colour. 
  For every $x$, we create a list of colours of
  cardinality $2$ from the list of $17$ colours.
  If $x^e= 1$, then the colours
  allowed at $x$ are the colours given to $T_1 x$ and $T_2 x$ (by membership
   in the Borel sets). 
  If $x^e= -1$, then the colours allowed at $x$
  are the colours given to $T^{-1} _1 x$ and $T^{-1} _2 x$.
  This structure follows that of Example  \ref{ex:i}, with the same
  concept of two potential arrows directed in two directions as determined
   by the $e$ coordinate.

       \vskip.2cm

       \noi  Now we define the adjacencies for the purpose of the
       list colouring.  The point $x$ is adjacent to
       $y$ if there is a $z$ such that an arrow could be pointed
       from $x$ to $z$ and an arrow could be pointed from $y$ to $z$
       as with Example \ref {ex:i}.  Notice that these points $x$ and $y$ differ by one of the $g_1, g_2, g_3, g_4$.   In this way, a   graph is
       created by  cliques such that for every  $z$ there is
       a clique of size equal to the p-degree of $z$, consisting of the vertices from which an arrow could be directed to $z$.  We call this clique
       the clique centered at $z$. 
      If $x$ is in $X\backslash X'$ then
      we assume it  has no adjacencies, so it is irrelevant which
       two colours are in the list. This graph of cliques we call the {\bf secondary}  graph. 
\end{ex}

       \noi Next we attempt to colour the vertices of this secondary graph 
 (of cliques)       using the same set of $17$ colours,
         forgetting how each point was coloured  originally, but  retaining
         the lists of colours and the
         above defined graph adjacencies from 
         cliques centered about points.

         \begin{props} \label {props:l}
           Proper list colouring of the secondary graph of  Example \ref{ex:l} is paradoxical.
         \end{props}

         \begin {proof} First assume we have a proper list colouring. 
       Because the original  colours are distinct
       on either side of the $g_1, \dots , g_4$,  there are two distinct
       colours in every  list, and they correspond to the two
       directions an arrow can be directed, in the same manner as with 
       Example \ref{ex:i}.  
       That the colouring is proper
       implies that there are no two arrows directed
       to the same point. If a list  colouring is proper and
       measurable (finitely additive $G$-invariant),
        there is one  arrow  coming out of every point in $X'$ 
        however coming  toward the  centres of cliques 
       the global average  is no more than $\frac {15} {16}$.
       As with Example \ref{ex:i}    each point can be  shifted according to
        its colour in the direction
       of its arrow to define 
        a measure preserving  injective map
       from all of the space to  a finite collection of  cylinders  of  measure
         $\frac {15}  {16}$. 

        Now we show that there are proper list  colourings of the secondary graph. 
Because  the colours are different on either side of the
length four  elements  $g_5, \dots g_{16}$,  if it is possible to
point an arrow from two distinct $x,y\in X'$ toward some $z$ but the choice
is made instead
 to point both of these arrows in the other directions,  $x$ and $y$ cannot
 get coloured with the same colour. (This is important, because they
 share a clique, the one centered at $z$.) 
 It follows from this that any orientation of arrows satisfying
  the colouring rule in Example
  \ref{ex:i} will
  also  define  a proper list colouring of the secondary graph. \end{proof} 
 
   \subsection { Proper colouring:}

       \vskip.2cm
\begin{ex} \label {ex:p}  
  We use the same structure of Borel sets and colours of
   Example  \ref{ex:l} and  assume that
        the number of sets/colours of that example
         is minimal, meaning that each set/colour  
       is given positive probability and, due to
       the ergodic property of each of the $g_1, \dots , g_{16}$,
       for every positive $\epsilon$  there is an odd  $N$ such that
       from all but a Borel subset $Q$ of measure no
       more than  $\epsilon$  for all $x$ in  $X\backslash Q$ all
       the  colours appear within a radius $N$ of $x$ when
 applying words of odd  length. (By ergodicity, this is true when restricted to just one of the 16 group elements $g_i$.) 
 We call this the original colouring of $X$. We create two copies of the space
  $X$, with the map $\rho$ defining the ``identity'' map 
  from the first copy to the second copy. If it helps to understand,
  we could say that each element of $G$ commnutes with $\rho$ and $\rho$ is a measure preserving involution between the two halves of the new space, each half given
   probability $\frac 12$.
  We keep the previous adjacency relations of the secondary graph of Example \ref{ex:l} 
  within the first copy of $X$, 
  meaning adjacencies  defined by
  cliques centred  about every $z$ of size equal to
  the p-degree of $z$.  In the second copy of $X$, for every $x$ 
  we connect with  edges the point  $\rho (x)$ with all the  $\rho (y)$
  such that  $y$ is 
  within a distance of $2N+10$ from $x$ applying  words of even length and
   $y$ coloured differently   from $x$ in the orignal colouring.  Furthermore
   any $x\in X\backslash Q$ in
   the first copy is connected to any $\rho (z)$ in the second copy
   if one moves from $ x$ to $z$ with a word of odd length
   of distance less than or equal to
    $N$  and $z$ is not coloured the same originally   as either 
    $z_1$ or $z_2$ such that
    there is a potential arrow from $x$ to the  $z_i$
    (according to the  contruction of Example \ref{ex:i}, which continues to
     define the structure of Example \ref{ex:l} and this example).
    If $x$ belongs to $Q$, then from $x$ in the first copy there
     are no adjacencies.

       \vskip.2cm

       \noi   Next we assume that $\epsilon$
       is strictly less than $\frac 1 {512}$,
       $N$ is so defined according to $\epsilon = \frac 1 {512}$,
       we drop the original colouring
       but keep the adjacencies that then define a graph with finite degree.
\end{ex} 

\begin{props} \label{props:p} 
  Proper colouring of Example \ref{ex:p} is paradoxical.
\end{props}

\begin{proof} 
  We assume that we  have a proper colouring of
  the two copies of $X$, using the same colours of the orignal colouring. If $z$ defines the centre of  a clique in the first
   copy,  these colours in the clique are distinct.  Furthermore, 
   with probability at least $1-\frac 1 {512}$ (membership
    in $X\backslash Q$) for an $x\in X$ the proper
   colouring at $x$ defines an arrow from $x$  to one of   
   two potential points  $z_1$ or $z_2$  
   according to the way the two points 
   $\rho (z_1)$ and $\rho (z_2)$ are now coloured. 
   Assuming a finitely additive $G$-invariant measure,
    the arrows defined   outward
    have a global
    weight  totaling  at least $1-\frac 1 {512}$. 
    On the other hand,
     given a $z$ defining a clique in the first copy of $X$, 
     the probability that at least one   point  in the clique
      centered at $z$ is  in $Q$ is no more than  $\frac 1 {128}$.   
      That means
      the weight of arrows toward such points
      $z$, on the average throughout $X$,
     can be no more than
    $1 + \frac 4 { 128} - \frac 1 {16}= \frac {31} {32}$.
   \vskip.2cm

   \noi   To show that there is some colouring according to the rule, we start by colouring the second copy of $X$  according to the original  colouring of the first copy  into Borel sets
    (via the map $\rho$).
    Because the set $Q$ has no adjacencies,
       any  previous solution from Example
       \ref{ex:l}  on the first copy of $X$ will    define a  proper colouring. 
 Hence the colouring is paradoxical.
  \end{proof}

 \section{Further study}

 {\bf Question 3:} Given  any probability space $X$ and a 
 finitely generated measure preserving group $G$ acting on $X$,
 let $P$ be the stationary  colouring rule
 that   requires only
 that the colouring must be proper, where adjacency is defined through
  the Cayley graph using the finitely many generators. Is there such an $X$ and $G$ such that the stationary  rule $P$ with finitely many   colours is paradoxical? \\

  Our first instinct is to believe that such a colouring rule $P$ cannot be paradoxical, 
  because the rule $P$ is not
   complex enough.
    Through increasing the number of colours, so one could think, 
   one jumps from impossible in any way to possible with respect to some finitely additive and invariant way.
   The situation is likely  very different to that
      with Borel colouring [7],
   where the Borel property separates   what can be accomplished
      with different numbers of colours.\\
  
    A  colouring rule with $k$ colours
     is (stationarily) \emph{essential} if
    it is paradoxical of rank  two and there exists no (stationary)
    paradoxical rank one colouring
    rule with the same  colours whose satisfaction
     implies the satisfaction of the original rule.
     A space $X$ is (stationarily) essential
      with $k$ colours
      if there is no (stationary) paradoxical rank one colouring  rule
       with $k$ colours defined
      on $X$, and
      yet there is  a paradoxical colouring  rule with $k$ colours
       defined on $X$  of  rank two. 
 Intriguing is the difference between what can be accomplished by
 stationary and non-stationary colouring rules. 
 Example \ref{ex:1}  shows that the Hausdorff rule is not stationarily essential. 
\\ 
  
\noi  {\bf Question 4:} Is 
 there a colouring rule that is stationarily essential?  
 \\
 
   We conjecture
 that the answer to   Question 4 is yes, though before
 communication with M. Bounds we thought that the
 Hausdorff rule was stationarily essential.\\ 

\noi {\bf Question 5:}    Does there exist a space and a semigroup action
    that is essential with respect to any number
    of colours? \\
    
When  none of the
    descendants are invertible Example \ref{ex:3}  doesn't apply, and  it is plausible that some paradoxical
    colouring rules are  not rank one  in character no matter how
     many colours are allowed.

    There is something satisfying about Example \ref{ex:i}  and unsatisfying
    about Examples \ref {ex:1}, \ref {ex:2},  and \ref{ex:3} .  Example  \ref{ex:i}  employs a
    stochastic process that seems to push colourings toward the satisfaction of
     the rule.
     On the other hand, the existence of  colourings witnessing the rules
     of  the first three examples 
      seem to be either accidental or  contrived.
     One could perceive
  satisfaction of a colouring rule to be a kind of fixed colouring, with
  the colouring rule defining some kind of iterative process that does or
   doesn't bring the colourings closer to satisfaction of the rule. 
   Of course if a colouring rule forced almost everwhere
   (with respect to the original measure $m$) an eventually stable 
    colour in finitely many colouring stages  with respect
    to  some initial measurable  colouring
    then there would be measurable
     colouring solutions with respect to the completion
     and the rule could not be paradoxical (given that the descendants remain measure
      preserving with respect to the completion). Although Example \ref{ex:i}  does
     seem to contain a force moving toward its satisfaction,
    iterations of the rule would likely involve long periods of stability followed by solutions 
    from one area of the space colliding occasionally with solutions from another. 
    A desired property may concern the relative stability of the colourings in the limit, and inspires the
       following question. \\

{\bf Question 6:}     Is there 
a  refinement  to  the  definition of a  rank one paradoxical colouring
 rule  
  that
    identifies a credible   force toward its satisfaction?
\\

Suppose one had a colouring rule for $S ^G$ where $S$ is a finite set,
$G$ is a group, and $C$ are the finite set of colours. 
  Another  structure to consider is $\{ S \times C\} ^G$,
  where we assume a random colouring start to $S^G$ inherited from
  the $C$ coordinate. We could consider how the colouring
   rule generates iterations of colourings on $\{ S \times C\} ^G$.

Like in Section 2 we will denote by
$\cal F$ the sigma algebra on which a
probability measure $m$ is defined, in most cases
${\cal F}$ will denote 
 Borel sets.
The \emph {deficiency} of a  colouring
        rule on $X$ is the infimum of all
        probabilities $\rho$ such that there is a
         ${\cal F}$ measurable  subset  
      of size  $1-\rho$ where 
      the  rule {\bf can}  be satisfied in some finitely
       additive and invariant way.
       \\
        
    \noi  {\bf Question 7:} Do all paradoxical colouring rules have positive deficiency?\\

        Assume  a sequence of colouring functions $c_1, c_2, \dots$,
        corresponding to finitely additive measures $\mu_1, \mu_2, \dots$
        extending the original measure (with the descendants measure preserving)
         and subsets $X_1, X_2, \dots$ such that for each $i=1,2,\dots$
        the measure of $X_i$ is greater than $1-\frac 1i$, 
        the $c_i$ are $\mu_i$ measurable, and 
        the colouring rule  is  satisfied by $c_i$ on $X_i$. Will there
        exist a colouring $c$ and a corresponding  finitely additive $\mu$ (with the
        descendants measure preserving)
        that satisfies the rule almost everywhere and is $\mu$ measurable?
        Initially it seems that the answer should be no, because there is no guarantee
        that any colouring in a pointwise limit should have any measurable properties. However
        we suspect  that the colouring rule will allow for finitely additive measurability in the limit, 
        meaning that the answer to Question 7 is yes and this
          is another difference between paradoxical
          colouring rules  and the topic of
          Borel colouring.
 
 The problem of extending the Cancellation Law inspires the
 following question.  \\

         \noi   {\bf Question 8:} Can one extend the conclusion of Theorem 1 to  $2[X] = [X]$?
\\
          
          In the definition of a colouring rule, we use that the
          descendants are measure preserving. This was a natural way
          to connect colouring rules to measure theoretic paradoxes.
          We could be more general in the application of
           an admissible relation. \\

           \noi 
{\bf Question 9:} 
Is there a reasonable and more general   definition for paradoxical
colouring rules  
           that  does not require that the descendants be   measure
           preserving?
           \\

\end{document}